\documentclass[english,letterpaper,11pt,reqno]{amsart}

\usepackage{appendix}
\usepackage{amsbsy}
\usepackage{amsfonts}
\usepackage{amsmath}
\usepackage{amssymb}
\usepackage{amsthm}
\usepackage{graphicx}
\usepackage{ifthen}
\usepackage{textcomp}
\usepackage{enumitem, color, amssymb}
\usepackage{dsfont}
\usepackage{mathrsfs}
\usepackage{calrsfs}
\usepackage[bookmarksnumbered,colorlinks]{hyperref}
\hypersetup{colorlinks=true, linkcolor=blue, citecolor=blue}
\newcommand{\lra}{\longrightarrow}

\newcommand{\RR}{\mathbb{R}}

\newcommand{\vep}{\varepsilon}

\makeatletter
\newcommand*{\defeq}{\mathrel{\rlap{%
                     \raisebox{0.3ex}{$\m@th\cdot$}}%
                     \raisebox{-0.3ex}{$\m@th\cdot$}}%
                     =}
\makeatother

\newtheorem{thm}{Theorem}
\newtheorem{lemma}{Lemma}
\newtheorem{cor}{Corollary}

\newtheorem{defn}{Definition}
\newtheorem{prop}{Proposition}
\newtheorem*{definition-non}{Definition}
\newtheorem*{theorem-non}{Theorem}
\newtheorem*{proposition-non}{Proposition}
\newtheorem*{lemma-non}{Lemma}
\newtheorem*{corollary-non}{Corollary}

\newcommand{\beqa}{\begin{eqnarray}}
\newcommand{\beq}{\begin{equation}}
\newcommand{\eeqa}{\end{eqnarray}}
\newcommand{\eeq}{\end{equation}}

\newcommand\ip[2]{g({#1},{#2})} %bracket scalar product
\newcommand\ipt[2]{\tilde{g}({#1},{#2})} %bracket scalar product
\newcommand\ii{i}

\newcommand\imp{\hspace{.2in}\Rightarrow\hspace{.2in}}

\newcommand\mb{\overline{\boldsymbol m}}
\newcommand\mm{{\boldsymbol m}}
\newcommand\kk{T}
\newcommand\xx{X}
\newcommand\yy{Y}
\newcommand\kkt{\widetilde{T}}
\newcommand\xxt{\widetilde{X}}
\newcommand\yyt{\widetilde{Y}}
\newcommand\cd[2]{\nabla_{\!#1}{#2}}
\newcommand\cdt[2]{\widetilde{\nabla}_{\!#1}{#2}}

\newcommand\gL{g_{\scriptscriptstyle L}}
\newcommand\gR{g_{\scriptscriptstyle R}}

\newcommand\comma{\hspace{.2in},\hspace{.2in}}
\newcommand\commas{\hspace{.1in},\hspace{.1in}}

\newtheorem{innercustomgeneric}{\customgenericname} %numbering theorems your own way
\providecommand{\customgenericname}{}
\newcommand{\newcustomtheorem}[2]{%
  \newenvironment{#1}[1]
  {%
   \renewcommand\customgenericname{#2}%
   \renewcommand\theinnercustomgeneric{##1}%
   \innercustomgeneric
  }
  {\endinnercustomgeneric}
}

\newcustomtheorem{customthm}{Theorem}
\newcustomtheorem{customlemma}{Lemma}

\begin{document}
\title[]{Killing vector fields on Riemannian and Lorentzian 3-manifolds}
\author[]{Amir Babak Aazami and Robert Ream}
\address{Clark University\hfill\break\indent
Worcester, MA 01610}
\email{aaazami@clarku.edu, rream@clarku.edu}

%%%%%%%%%%%
\begin{abstract}
We give a complete local classification of all Riemannian 3-manifolds $(M,g)$ admitting a nonvanishing Killing vector field $T$.  We then extend this classification to timelike Killing vector fields on Lorentzian 3-manifolds, which are automatically nonvanishing.  The two key ingredients needed in our classification are the scalar curvature $S$ of $g$ and the function $\text{Ric}(T,T)$, where $\text{Ric}$ is the Ricci tensor; in fact their sum appears as the Gaussian curvature of the quotient metric obtained from the action of $T$.  Our classification generalizes that of Sasakian structures, which is the special case when $\text{Ric}(T,T) = 2$.  We also give necessary, and separately, sufficient conditions, both expressed in terms of $\text{Ric}(T,T)$, for $g$ to be locally conformally flat.  We then move from the local to the global setting, and prove two results: in the event that $T$ has unit length and the coordinates derived in our classification are globally defined on $\RR^3$, we give conditions under which $S$ completely determines when the metric will be geodesically complete.  In the event that the 3-manifold $M$ is compact, we give a condition stating when it admits a metric of constant positive sectional curvature.
\end{abstract}

\maketitle

%%%%%%%%%%%
\section{Introduction}
The aim of this paper is to give a complete local classification of all Riemannian 3-manifolds $(M,g)$ that admit a nonvanishing Killing vector field $T$.  In fact this classification will also yield a  related one: that of all \emph{Lorentzian} 3-manifolds supporting a \emph{timelike} Killing vector field.  Our classification proceeds by considering the special case when $T$ has constant length: the general case follows from this one by applying a conformal change by the factor of $g(T,T$).  But in fact there are important reasons for imposing this condition. One of them is that constant length allows us to adapt the machinery of the \emph{Newman-Penrose formalism} \cite{newpen62}\,---\,a construct that originated in 4-dimensional Lorentzian geometry\,---\,to the setting of 3-dimensional Riemannian geometry; see also \cite{schmidt2014,NC,bettiol2018}, wherein similar frame techniques have been applied in dimension 3.

As shown in Section \ref{sec:NP} below, constant length is a prerequisite for this formalism.  But more importantly, our interest in constant length arises from what we regard as the ``canonical" constant length Killing vector field in dimension 3: the unit length Killing vector field $T$ on $(\mathbb{S}^3,\mathring{g})$ tangent to the Hopf fibration, where $\mathring{g}$ is the standard (round) metric.  Given the special geometry of $(\mathbb{S}^3,\mathring{g})$  as a spherical space form, and the presence of such a vector field on it, we take our motivation from the following questions:
\begin{enumerate}[leftmargin=*]
\item[1.] Can one classify locally all Riemannian 3-manifolds admitting a constant length Killing vector field? Do they take on a ``canonical" form?
\item[2.] If a Riemannian 3-manifold admits a constant length Killing vector field, then when will it  be locally conformally flat, as with $(\mathbb{S}^3,\mathring{g})$?
\end{enumerate}
(Yet another path of inquiry, which we do not pursue here, would be to examine when the circle action provided by a constant length Killing vector field is free, and the role that sectional curvature plays in this; see \cite{BN}.) A complete answer to our first question above is provided in our first Theorem:

\begin{thm}
\label{thm:main}
Let $(M,g)$ be a Riemannian 3-manifold that admits a unit length Killing vector field $T$.  Then there exist local coordinates $(t,r,\theta)$ and a smooth function $\varphi(r,\theta)$ such that
\beqa
\label{eqn:gT0}
T = \partial_t \comma g = (T^{\flat})^2 + dr^2 + \varphi^2d\theta^2,
\eeqa
and where the quotient metric $dr^2 + \varphi^2d\theta^2$ has Gaussian curvature
\beqa
\label{eqn:gauss0}
-\frac{\varphi_{rr}}{\varphi} = \frac{1}{2}\big(S + \emph{\text{Ric}}(T,T)\big),
\eeqa
with $S$ and $\emph{\text{Ric}}$ the scalar curvature and Ricci tensor of $g$, respectively.  If \eqref{eqn:gT0} is given globally on $M = \RR^3$ \emph{(}with $r,\theta$ polar coordinates on $\RR^2$\emph{)}, and if $\emph{\text{Ric}}(T,T) =  0, \varphi(0,\theta) = 0$, $\varphi_r(0,\theta)=1$, $\varphi(-r,\theta+\pi) = -\varphi(r,\theta)$, and $\varphi > 0$ when $r > 0$, then $g$ is complete if and only if
$$
\lim_{r\to \infty} \inf_{|p| \geq r} S\big|_p\leq 0.
$$
\end{thm}

The following remarks help to shed light on this result:
\begin{enumerate}[leftmargin=*]
\item[i.] After our first preprint appeared, we learned of the works \cite{manzano14,manzano17}, in which the existence of coordinates isometric to \eqref{eqn:gT0} are proved, as well as a result that includes \eqref{eqn:gauss0} as a special case; these were obtained via a different method than ours, and applied to the classification of Riemannian submersions from 3-manifolds to a surface, whose fibers are the integral curves of a Killing vector field.
\item[ii.] An almost identical Theorem exists for unit \emph{timelike} Killing vector fields on \emph{Lorentzian} 3-manifolds; see Corollary \ref{thm:main2} in Section \ref{sec:Lor} below, wherein the relevant Lorentzian terminology is also defined.
\item[iii.] If $T$ does not have unit length, then \eqref{eqn:gT0} is scaled by $g(T,T)$.  Also, in the global case $M = \RR^3$, completeness can be obtained by conditions other than $\text{Ric}(T,T) = 0$, using, e.g., \cite[Theorem~6]{MaschlerReam}.
\item[iv.] The ``canonical form" alluded to above is manifested in \eqref{eqn:gT0} and \eqref{eqn:gauss0}; for the form of the metric in the coordinate basis $\{\partial_t,\partial_r,\partial_\theta\}$, see \eqref{eqn:gtr0} in Section \ref{sec:proof} below.  As Theorem \ref{thm:main} makes clear, our classification depends entirely on two functions, the scalar curvature $S$ and $\text{Ric}(T,T)$.  Let us say more about the latter function, which is especially important; one way to appreciate its significance when $\text{dim}\,M = 3$ is as follows.  If a vector field $T$ has constant length and geodesic flow (as does any constant length Killing vector field), then the function $\text{Ric}(T,T)$, if nonnegative, completely governs whether its orthogonal complement $T^{\perp} \subseteq TM$ is integrable.  As a consequence, it was shown in \cite{hp13} that when such a $T$ satisfies $\text{Ric}(T,T) > 0$ and when $M$ is orientable and compact, then $T^{\flat}$ is a contact form and $T$ is its Reeb vector field; if in addition $T$ is divergence-free and $\text{Ric}(T,T) = 1$, then $T_*(T^{\perp})$ is $J$-invariant, where $J$ is the Levi-Civita almost-complex structure on $TTM$ (in fact  these two conditions are necessary and sufficient).
\item[v.] The previous remark did not assume that  $T$ is a unit length Killing vector field.  Imposing this condition\,---\,as well as the condition $\text{Ric}(T,T) = 2$, so that the endomorphism in \eqref{eqn:DT} below defines an almost complex structure on $T^{\perp}$\,---\,would make $(M,g,T)$ a \emph{Sasakian structure}. In dimension 3, a classification of these on closed manifolds was obtained in \cite{geiges97}, up to diffeomorphism; an  explicit metric classification was then given in \cite{belgun01,belgun03}, which also established a one-to-one correspondence between Sasakian and \emph{normal CR structures}, and also classified the latter.  For an application to monopole fields, see \cite{harris}.
\end{enumerate}

Given the importance of the quantities $S$ and $\text{Ric}(T,T)$, it is worthwhile to mark the following relationship between unit length Killing vector fields and constant curvature in the compact setting, a result which is essentially a corollary of a well known result in \cite{Hamilton}:
\begin{thm}
\label{thm:2}
Let $(M,g)$ be a compact Riemannian 3-manifold and $T$ a globally defined, unit length Killing vector field.  If $\emph{\text{Ric}}(T,T) \neq 0$ and
\beqa
\label{eqn:pos}
S > 2\frac{|\emph{\text{Ric}}(T)|_g^2}{\emph{\text{Ric}}(T,T)}-\emph{\text{Ric}}(T,T),
\eeqa
where $\emph{\text{Ric}}(T)$ is the Ricci operator, then $M$ admits a metric of constant positive sectional curvature.
\end{thm}

Finally, our answer to the second question is also given in terms of $\text{Ric}(T,T)$:

\begin{thm}
\label{thm:3}
Let $(M,g)$ be a Riemannian 3-manifold that admits a unit length Killing vector field $T$.  If $g$ is locally conformally flat, then
\beqa
\label{eqn:wPDE}
4|\emph{\text{Ric}}(T)|_g^2 = 3\emph{\text{Ric}}(T,T)^2 - 2B\emph{\text{Ric}}(T,T) + C
\eeqa
for some constants $B,C$, where $\emph{\text{Ric}}(T)$ is the Ricci operator.
Conversely, given \eqref{eqn:wPDE}, there exist coordinates $(r,\theta)$ on the quotient metric in \eqref{eqn:gT0} with respect to which $g$ is conformally flat when 
\[ \omega_\theta=0 \comma \omega_r^2 + \frac{1}{4}(\omega^2 + 2B)^2= C+B^2 \comma \varphi=h(\theta)\omega_r, \]
where $\omega^2  = 2\emph{\text{Ric}}(T,T)$, $\varphi$ is as in Theorem \ref{thm:main}, and $h(\theta)$ is a smooth function.  If $\emph{\text{Ric}}(T,T)$ is constant, then  $g$ is locally conformally flat if and only if $S = 3\emph{\text{Ric}}(T,T)$.
\end{thm}

As a check of the last statement, note that for $(\mathbb{S}^3,\mathring{g})$ with Hopf Killing vector field $T$ and radius $R$,
$$
\text{Ric}(T,T) = \frac{2}{R^2} \comma S = \frac{6}{R^2}\cdot
$$

%%%%%%%%%%%%%%%%%
\section{Divergence, Twist, Shear}
\label{sec:NP}
A Killing vector field $T$ on a Riemannian manifold $(M,g)$ is defined by the condition
$
\label{eqn:KVF}
\mathfrak{L}_T g = 0,
$
where $\mathfrak{L}$ is the Lie derivative.  However, when $T$ has unit length and when $\text{dim}\,M = 3$, there is an equivalent formulation, given by Lemma \ref{lemma:KVF} below, which plays a crucial role in our classification.  This formulation also involves the Lie derivative, but owing to the low dimension, only certain components of it, which components carry geometric properties of the flow of $T$.  These properties are the \emph{divergence}, \emph{twist}, and \emph{shear}; as the latter two are not as well known as the former, we now digress to define them explicitly.  Thus, let $\kk$ be a smooth unit length vector field defined in an open subset of a Riemannian 3-manifold $(M,g)$, so that $\cd{v}{\kk} \perp \kk$ for all vectors $v$ ($\nabla$ is the Levi-Civita connection).  Let $\xx$ and $\yy$ be two smooth vector fields such that $\{\kk,\xx,\yy\}$ is a local orthonormal frame.  Now define the following endomorphism $D$ of the orthogonal complement $\kk^{\perp} \subseteq TM$,
\beqa
\label{eqn:DT}
D\colon \kk^{\perp} \lra \kk^{\perp}\hspace{.2in},\hspace{.2in}v\ \mapsto\ \cd{v}{\kk},
\eeqa
and observe that its matrix with respect to the frame $\{\kk,\xx,\yy\}$ is
\beqa
\label{eqn:matrix2}
D = \begin{pmatrix}
        \ip{\cd{\xx}{\kk}}{\xx} & \ip{\cd{\yy}{\kk}}{\xx}\\
        \ip{\cd{\xx}{\kk}}{\yy} & \ip{\cd{\yy}{\kk}}{\yy}
    \end{pmatrix}\cdot\nonumber
\eeqa
Contained within this matrix are three geometric properties associated to the flow of $T$:
\begin{enumerate}[leftmargin=*]
\item[1.] The divergence of $\kk$, denoted $\text{div}\,\kk$, is simply the trace of $D$.
\item[2.] By Frobenius's theorem, $\kk^{\perp}$ is integrable if and only if the anti-symmetric part of $D$ vanishes; as seen in \eqref{eqn:matrix} below, this vanishing is governed by the following function, which comprises the off-diagonal elements of the anti-symmetric part of $D$:
\beqa
\label{eqn:rotation}
\omega \defeq \ip{\kk}{[\xx,\yy]} = \ip{\cd{\yy}{\kk}}{\xx} - \ip{\cd{\xx}{\kk}}{\yy}.
\eeqa
Since $\omega^2$ equals the determinant of the anti-symmetric part of $D$, it is a frame independent quantity.  We call $\omega^2$ the \emph{twist function} of $T$ and say that the flow of $\kk$ is \emph{twist-free} if $\omega^2 = 0$.
\item[3.] The third piece of information is the \emph{shear} $\sigma$ of $\kk$; it is given by the trace-free symmetric part of $D$, whose components $\sigma_1,\sigma_2$ we combine here into a complex-valued quantity:
\beqa
\hspace{.26in}\sigma\!\!&\defeq&\!\! \underbrace{\,\frac{1}{2}\Big(\ip{\cd{\yy}{\kk}}{\yy} - \ip{\cd{\xx}{\kk}}{\xx}\Big)\,}_{\sigma_1} +\ \ii\underbrace{\,\frac{1}{2}\Big(\ip{\cd{\yy}{\kk}}{\xx} + \ip{\cd{\xx}{\kk}}{\yy}\Big)\,}_{\sigma_2}.\nonumber\\
&&\label{eqn:shear}
\eeqa
Although $\sigma$ itself is not frame independent, its magnitude $|\sigma^2|$ is: by \eqref{eqn:matrix} below, it is minus the determinant of the trace-free symmetric part of $D$.  We say that the flow of $\kk$ is \emph{shear-free} if $\sigma = 0$.  As with being twist-free, being shear-free is a frame independent statement. In terms of $\text{div}\,\kk$, $\omega$, and $\sigma$, $D$ takes the form
\beqa
\label{eqn:matrix}
D = \begin{pmatrix}\frac{1}{2}\text{div}\,T &0\\0 & \frac{1}{2}\text{div}\,T\end{pmatrix}\ + 
\!\!\!\underbrace{\,\begin{pmatrix}
        - \sigma_1 & \sigma_2\\
        \sigma_2 & \sigma_1
    \end{pmatrix}\,}_{\text{trace-free symmetric}}\!\!\! +\, \underbrace{\,\begin{pmatrix}
        0 & \frac{\omega}{2}\\
        - \frac{\omega}{2} & 0
    \end{pmatrix}\,}_{\text{anti-symmetric}}\cdot
\eeqa
\end{enumerate}
We record here the well known fact that, while being divergence-free is not a conformal invariant, being shear-free or twist-free \emph{is}:

\begin{lemma}
Let $(M,g)$ be a Riemannian 3-manifold and $T$ a unit length vector field.  Given a conformal metric $\tilde{g} = e^{2f}g$, the vector field $\widetilde{T} = e^{-f}T$ is shear-free with respect to $\tilde{g}$ if and only if $T$ is shear-free with respect to $g$. Likewise if shear-free is replaced by twist-free.
\end{lemma}

\begin{proof}
Note that $\tilde{g}(\widetilde{T},\widetilde{T}) = 1$; given a $g$-orthonormal frame $\{T,X,Y\}$, form the $\tilde{g}$-orthonormal frame $\{\widetilde{T},\widetilde{X},\widetilde{Y}\}$, with $\widetilde{X} = e^{-f}X$ and $\widetilde{Y} = e^{-f}Y$.  Denoting by $\widetilde{\nabla}$ the Levi-Civita connection of $\tilde{g}$, the shear of $\widetilde{T}$ with respect to $\tilde{g}$ is
\beqa
\tilde{\sigma} \!\!&=&\!\! \frac{1}{2}\Big(\ipt{\cdt{\yyt}{\kkt}}{\yyt} - \ipt{\cdt{\xxt}{\kkt}}{\xxt}\Big) + \frac{i}{2}\Big(\ipt{\cdt{\yyt}{\kkt}}{\xxt} + \ipt{\cdt{\xxt}{\kkt}}{\yyt}\Big)\nonumber\\
\!\!&=&\!\! \frac{e^{-f}}{2}\Big(\ip{\cdt{\yy}{\kk}}{\yy} - \ip{\cdt{\xx}{\kk}}{\xx}\Big) + i\frac{e^{-f}}{2}\Big(\ip{\cdt{\yy}{\kk}}{\xx} + \ip{\cdt{\xx}{\kk}}{\yy}\Big)\nonumber\\
\!\!&=&\!\! e^{-f}\sigma,\phantom{\frac{1}{2}}\nonumber
\eeqa
where in the last step we have used standard formulae relating Levi-Civita connections of conformal metrics, e.g.,
$$
\cdt{Y}{T} = \cd{Y}{T} + Y(f)T + T(f)Y.
$$
Likewise for the twist:
$$
\tilde{\omega} = \tilde{g}(\kkt,[\xxt,\yyt]) = e^{-f}g(T,[X,Y]) = e^{-f}\omega. 
$$
Thus not only twist-free, but  shear-free as well, is a conformal property: $|\tilde{\sigma}|^2  = 0 \Leftrightarrow |\sigma|^2 = 0$ and $\tilde{\omega}^2 = 0 \Leftrightarrow \omega^2 = 0$.
\end{proof}

When $\text{dim}\,M \geq 4$, divergence and shear alone are not enough to characterize unit length Killing vector fields, but they do when $\text{dim}\,M = 3$:

\begin{lemma}
\label{lemma:KVF}
A unit length vector field $T$ on a Riemannian 3-manifold $(M,g)$ is a Killing vector field if and only if its flow is geodesic, divergence-free, and shear-free.
\end{lemma}

\begin{proof}
The Killing condition is equivalent to
\beqa
\label{eqn:KVF2}
g(\cd{v}{T},w) + g(\cd{w}{T},v) = 0 \hspace{.2in}\text{for all}~v, w \in TM,
\eeqa
from which it follows that any Killing vector field $T$ is divergence-free and shear-free, via \eqref{eqn:shear}.  Finally, \eqref{eqn:KVF2} also implies that any unit length Killing vector field must have geodesic flow: 
$$
\cd{T}{T} = 0.
$$
Conversely, suppose that a unit length vector field $T$ is geodesic, divergence-free, and shear-free, and consider \eqref{eqn:KVF2}.   Writing $v,w$ with respect to an orthonormal frame $\{T,X,Y\}$ as
$$
v = a_0T+a_1X+a_2Y  \comma w = b_0T+b_1X+b_2Y,
$$
we have
\beqa
g(\cd{v}{T},w) + g(\cd{w}{T},v) \!\!&=&\!\! a_1g(\cd{X}{T},w) + a_2g(\cd{Y}{T},w)\nonumber\\
&&\hspace{.2in} +\ a_1g(\cd{w}{T},X) + a_2g(\cd{w}{T},Y)\nonumber\\
&=&\!\!  2a_1b_1g(\cd{X}{T},X) + (a_1b_2 + b_1a_2)g(\cd{X}{T},Y)\nonumber\\
&&\hspace{.2in} +\ (a_1b_2 + b_1a_2)g(\cd{Y}{T},X) +  2a_2b_2\underbrace{\,g(\cd{Y}{T},Y)\,}_{g(\cd{X}{T},X)}\nonumber\\
&=&\!\! 2(a_1b_1+a_2b_2)\underbrace{\,g(\cd{X}{T},X)\,}_{\frac{1}{2}\text{div}\,T}\nonumber\\
&&\hspace{.2in} +\ (a_1b_2 + b_1a_2)\underbrace{\,\big(g(\cd{X}{T},Y)+g(\cd{Y}{T},X)\big)\,}_{2\sigma_2}.\nonumber
\eeqa
This vanishes by our assumptions, completing the proof.
\end{proof}

We can now state our plan of attack: divergence, geodesic flow, twist, and shear all involve \emph{first} derivatives of $T$, whereas curvature involves \emph{second} derivatives.  Our plan of attack, therefore, is to express the components of the Riemann curvature tensor in terms of the divergence, twist, and shear of $T$, \emph{thereby reducing second-order equations to first-order ones}\,---\,indeed, further encouraged by the fact that, as we have just seen, if $T$ is a unit length Killing vector field,  then $\text{div}\,T, \sigma$, and $\cd{T}{T}$ all vanish, so that only $T$'s twist function $\omega^2$ is unknown.  The hope is that this will simplify things enough to allow a full determination of the metric.   And it will\,---\,after we express the curvature in terms of the divergence, twist, and shear, which we now proceed to do.

%%%%%%%%%%%
\section{The Newman-Penrose Formalism for Riemannian  3-manifolds}
\label{sec:NP2}

In what follows we present the Newman-Penrose formalism for Riemannian 3-manifolds, presenting here only the resulting  equations; complete derivations can be found in \cite{AA14}.  Let $\{\kk,\xx,\yy\}$ be an orthonormal frame\,---\,with $T$ not necessarily a Killing vector field\,---\,and form the complex-valued quantities
\beqa
\label{eqn:complex}
\mm \defeq \frac{1}{\sqrt{2}}(\xx-\ii \yy)\hspace{.2in},\hspace{.2in}\mb \defeq \frac{1}{\sqrt{2}}(\xx+\ii \yy).\nonumber
\eeqa
Henceforth we work with the complex frame $\{\kk,\mm,\mb\}$, for which only
$$
\ip{\kk}{\kk} = 1 \comma \ip{\mm}{\mb} = 1
$$
are nonzero.  The following quantities associated to this complex frame play a central role in all that follows.

\begin{defn}
\label{def:spin}
The {\rm spin coefficients} of the complex frame $\{\kk,\mm,\mb\}$ are the complex-valued functions
\beqa
\label{eqn:sc}
\kappa\!\!\! &=&\!\!\! -\ip{\cd{\kk}{\kk}}{\mm}\hspace{.2in},\hspace{.2in}\rho=-\ip{\cd{\mb}{\kk}}{\mm}\hspace{.2in},\hspace{.2in}\sigma=-\ip{\cd{\mm}{\kk}}{\mm},\nonumber\\
&&\hspace{.45in}\vep=\ip{\cd{\kk}{\mm}}{\mb}\hspace{.2in},\hspace{.2in}\beta=\ip{\cd{\mm}{\mm}}{\mb}.\nonumber
\eeqa
\end{defn}
Note that, because $T$ has unit length, its flow is geodesic, $\cd{T}{T} =  0$, if and only if $\kappa = 0$; that $\sigma$, when written out in terms of its real and imaginary parts, is precisely the complex shear \eqref{eqn:shear}; and that the spin coefficient $\rho$ has real and imaginary parts given by
\beqa
\label{eqn:rho2}
\rho = -\frac{\text{div}\,\kk}{2} - \ii\, \frac{\omega}{2}\cdot
\eeqa
In other words, the first three spin coefficients $\kappa,\rho,\sigma$ stand in for the geometric properties of the flow of $\kk$ discussed above.  In terms of all five spin coefficients, the Lie brackets are
\beqa
[\kk,\mm] \!\!&=&\!\! \kappa\,\kk + (\vep + \bar{\rho})\,\mm + \sigma\,\mb,\label{eqn:LB1}\\ 
\,[\mm,\mb] \!\!&=&\!\! (\bar{\rho} - \rho)\,\kk + \bar{\beta}\,\mm - \beta\,\mb.\label{eqn:LB2}
\eeqa
(The remaining Lie bracket $[\kk,\mb]$ is obtained by complex conjugation.)  Now to the curvature; to begin with, our sign convention for the Riemann curvature tensor is
$$
R(X,Y)Z = \cd{X}{\cd{Y}{Z}} - \cd{Y}{\cd{X}{Z}} - \cd{[X,Y]}{Z},\nonumber\\
$$
in which case the Ricci tensor with respect to the complex frame $\{\kk,\mm,\mb\}$ is
$$
\text{Ric}(v,w) = R(\kk,v,w,\kk) + R(\mm,v,w,\mb) + R(\mb,v,w,\mm).
$$
The following identities satisfied by the Ricci tensor in the complex frame $\{T,\mm,\mb\}$ will appear in formulae below:
$$
\left\{
\begin{array}{rcl}
\text{Ric}(\mm,\mm) \!\!&=&\!\! R(\kk,\mm,\mm,\kk),\nonumber\\
\text{Ric}(\kk,\kk) \!\!&=&\!\! 2R(\mm,\kk,\kk,\mb),\nonumber\\
\text{Ric}(\kk,\mm) \!\!&=&\!\! R(\mm,\kk,\mm,\mb),\nonumber\\
\text{Ric}(\mm,\mb) \!\!&=&\!\! \frac{1}{2}\text{Ric}(\kk,\kk) + R(\mb,\mm,\mb,\mm).\nonumber
\end{array}
\right.
$$
The Newman-Penrose begins by expressing the Lie brackets in terms of spin coefficients, as we saw in \eqref{eqn:LB1} and \eqref{eqn:LB2} above.  It then moves down to the level of curvature, by expressing the following components of the curvature tensor,
\beqa
R(T,\mb,T,\mm) &\commas& R(T,\mm,T,\mm) \comma R(\mb,\mm,T,\mm)\nonumber\\
&&\hspace{-.75in} R(T,\mm,\mm,\mb) \comma R(\mb,\mm,\mm,\mb),\nonumber
\eeqa
in terms of the Ricci tensor and the spin coefficients.  Doing so, the following (first-order) equations arise; they play the driving role in our classification:
\beqa
\label{eqn:Sachs1}
\kk(\rho) - \mb(\kappa) \!\!&=&\!\! |\kappa|^2 + |\sigma|^2 + \rho^2 + \kappa\bar{\beta} + \frac{1}{2} {\rm Ric}(\kk,\kk),\label{eqn:S1}\\
\kk(\sigma) - \mm(\kappa) \!\!&=&\!\! \kappa^2 + 2\sigma\vep + \sigma(\rho + \bar{\rho}) - \kappa \beta + {\rm Ric}(\mm,\mm),\phantom{\frac{1}{2}}\label{eqn:S2}\\
\mm(\rho) - \mb(\sigma) \!\!&=&\!\! 2 \sigma\bar{\beta} + (\bar{\rho}-\rho)\kappa + {\rm Ric}(\kk,\mm),\phantom{\frac{1}{2}}\label{eqn:S3}\\
\kk(\beta) - \mm(\vep) \!\!&=&\!\! \sigma(\bar{\kappa} - \bar{\beta}) + \kappa (\vep - \bar{\rho}) + \beta(\vep + \bar{\rho}) - {\rm Ric}(\kk,\mm),\phantom{\frac{1}{2}}\label{eqn:S5}\\
\mm(\bar{\beta}) + \mb(\beta) \!\!&=&\!\! |\sigma|^2 - |\rho|^2 -2|\beta|^2 + (\rho - \bar{\rho})\vep - {\rm Ric}(\mm,\mb) + \frac{1}{2} {\rm Ric}(\kk,\kk).\nonumber\\\label{eqn:S4}
\eeqa

Finally, up to complex conjugation, there are two nontrivial differential Bianchi identities:
\beqa
&&\hspace{-.4in}\kk({\rm Ric}(\kk,\mm))\, -\, \frac{1}{2}\mm({\rm Ric}(\kk,\kk)) + \mb({\rm Ric}(\mm,\mm))=\nonumber\\
&&\hspace{-.1in}\kappa\,\big({\rm Ric}(\kk,\kk) - \text{Ric}(\mm,\mb)\big) + \big(\vep + 2\rho + \bar{\rho}\big){\rm Ric}(\kk,\mm)\phantom{\frac{1}{2}}\label{eqn:bid}\\
&&\hspace{1.3in}+\, \sigma\,{\rm Ric}(\kk,\mb)- \big(\bar{\kappa} + 2\bar{\beta}\big){\rm Ric}(\mm,\mm)\phantom{\frac{1}{2}}\nonumber\\
\text{and}&&\nonumber\\
&&\hspace{-.4in}\mm({\rm Ric}(\kk,\mb)) + \mb({\rm Ric}(\kk,\mm)) - \kk\big({\rm Ric}(\mm,\mb) - (1/2){\rm Ric}(\kk,\kk)\big) =\phantom{\frac{1}{2}}\nonumber\\
&&\hspace{-.1in}(\rho+\bar{\rho})\big({\rm Ric}(\kk,\kk) - {\rm Ric}(\mm,\mb)\big) - \bar{\sigma}{\rm Ric}(\mm,\mm) - \sigma{\rm Ric}(\mb,\mb)\phantom{\frac{1}{2}}\label{eqn:bid2}\\
&&\hspace{1.2in} -\, \big(2\bar{\kappa} + \bar{\beta}\big){\rm Ric}(\kk,\mm) - \big(2\kappa + \beta\big){\rm Ric}(\kk,\mb).\phantom{\frac{1}{2}}\nonumber
\eeqa

We now immediately specialize to the case when $T$ is a Killing vector field:

\begin{lemma}
\label{lemma:KVF2}
Let $(M,g)$ be a Riemannian 3-manifold admitting a unit length Killing vector field $T$ with twist function $\omega^2$.  With respect to any complex frame $\{T,\mm,\mb\}$, the Ricci tensor $\emph{Ric}$ and scalar curvature $S$ satisfy
\beqa
T(\omega) \!\!\!&=&\!\!\! 0 \comma \emph{\text{Ric}}(T,T) = \frac{\omega^2}{2} \comma \emph{\text{Ric}}(\mm,\mm) = 0,\nonumber\\
&&\hspace{-.3in}\mm(\bar{\beta}) + \mb(\beta) = -2|\beta|^2 -i\omega\vep - \frac{1}{2}\Big(S - \frac{\omega^2}{2}\Big)\cdot\label{eqn:Ric0}
\eeqa
When \eqref{eqn:Ric0} is written in terms of the underlying orthonormal frame $\{T,X,Y\}$ of the complex frame, it is
\beqa
\label{eqn:Ric}
X(\emph{\text{div}}\,X) + Y(\emph{\text{div}}\,Y) = -(\emph{\text{div}}\,X)^2 - (\emph{\text{div}}\,Y)^2 -i\omega\vep -\frac{1}{2}\Big(S - \frac{\omega^2}{2}\Big)\cdot
\eeqa
\end{lemma}

\begin{proof}
By Lemma \ref{lemma:KVF}, we know that $$\kappa = \sigma = \rho + \bar{\rho} = 0;$$ inserting these into \eqref{eqn:S1} and \eqref{eqn:S2} directly yields the first line of equations in the  statement of Lemma \ref{lemma:KVF2}; e.g., 
$$
\text{Ric}(T,T) = \frac{\omega^2}{2} \comma T(\omega) = 0,
$$
are, respectively, the real and imaginary parts of \eqref{eqn:S1}.  Meanwhile, \eqref{eqn:Ric0} follows from \eqref{eqn:S4}, which has no imaginary part, and the fact that  the scalar curvature $S$ in terms of the complex frame $\{T,\mm,\mb\}$ is
\beqa
\label{eqn:S}
S = \text{Ric}(T,T) + 2\text{Ric}(\mm,\mb) \imp  \text{Ric}(\mm,\mb) = \frac{S}{2} - \frac{\omega^2}{4}\cdot
\eeqa
Finally, \eqref{eqn:Ric} follows from the fact that, when $\cd{T}{T} = 0$,
\beqa
\beta = \frac{1}{\sqrt{2}}\big(\ip{\cd{\yy}{\xx}}{\yy} + i \ip{\cd{\xx}{\xx}}{\yy}\big) = \frac{1}{\sqrt{2}} (\text{div}\,\xx - i\,\text{div}\,\yy),\label{eqn:beta}
\eeqa
which completes the proof.
\end{proof}

We have not yet considered the differential Bianchi identities; let us do so now.  Inserting the contents of Lemma \ref{lemma:KVF2} into \eqref{eqn:bid} and \eqref{eqn:bid2}, as well as $\bar{\rho} = -\rho$, yields
$$
\kk(\mm(\rho)) = (\vep + \bar{\rho})\mm(\rho)
$$
for \eqref{eqn:bid}; but this is precisely the Lie bracket \eqref{eqn:LB1} applied to $\rho$ (bearing in mind that $T(\rho) = 0$), and therefore carries no new information. As for \eqref{eqn:bid2}, it yields
$$
-\mm(\mb(\rho)) + \mb(\mm(\rho))  = - \bar{\beta}\,\mm(\rho) + \beta\,\mb(\rho),
$$
where $T(S) = 0$ and $\bar{\rho} = -\rho$ have been used.  But this is precisely the Lie bracket \eqref{eqn:LB2} applied to $\rho$, so that \eqref{eqn:bid2} also yields no new information.  

\section{Local Coordinates}
The goal of this section is to establish the ``right" local coordinates in which to prove Theorem \ref{thm:main} in the next section.  To begin with, recall that because
$$
\kappa = \sigma = \rho + \bar{\rho} = 0,
$$
the only spin coefficients remaining are $\vep$ and $\beta$.  Observe that the former is in fact purely imaginary,
\beqa
\label{eqn:imag}
\vep = i\ip{\cd{\kk}{\xx}}{\yy},
\eeqa
and the latter, when $\cd{T}{T} = 0$, is given by \eqref{eqn:beta}.   The following ``gauge freedom" simultaneously enjoyed by these two spin coefficients will prove useful in the proof of Theorem \ref{thm:main}:

\begin{prop}
\label{prop:boost}
Let $T$ be a unit length Killing vector field with twist function $\omega^2$ and $\{\kk,\mm,\mb\}$ a complex frame.  Then there exists a smooth real-valued function $\vartheta$ such that the complex frame $\{T,\mm_*,\mb_*\}$ defined by the rotation
$$
\mm_* \defeq e^{i\vartheta}\mm\comma\mb_* \defeq e^{-i\vartheta}\mb
$$
has spin coefficients $\kappa_* = \sigma_* = 0, \rho_* = \rho$,
\beqa
\label{eqn:guage}
\vep_* = \rho \comma \emph{\text{Re}}(\beta_*) = 0 \comma T(\beta_*) = 0.
\eeqa
\end{prop}

\begin{proof}
By definition, 
$$
\kappa_* = -\ip{\cd{\kk}{\kk}}{\mm_*} = e^{i\vartheta}\kappa = 0;
$$
similarly, $\sigma_* = e^{2i\vartheta}\sigma = 0$, and $\rho_* = \rho$ (in particular, $\omega_*^2 = \omega^2$).  Next,
\beqa
\vep_* \!\!&=&\!\! \ip{\cd{\kk}{\mm_*}}{\mb_*}\nonumber\\
&=&\!\! e^{-i \vartheta}\ip{\cd{\kk}{(e^{i\vartheta}\mm)}}{\mb}\nonumber\\
&=&\!\! \vep + e^{-i\vartheta}\kk(e^{i \vartheta})\nonumber\\
&=&\!\! \vep + i\kk(\vartheta).\label{eqn:gauge2}
\eeqa
Similarly,
\beqa
\beta_* \!\!&=&\!\! \ip{\cd{\mm_*}{\mm_*}}{\mb_*}\nonumber\\
&=&\!\!\ip{\cd{\mm}{(e^{i\vartheta}\mm)}}{\mb}\nonumber\\
&=&\!\! e^{i\vartheta}( \beta + i\mm(\vartheta)).\nonumber
\eeqa
By \eqref{eqn:imag} and \eqref{eqn:gauge2}, we may choose a locally defined function $\vartheta$ so that
$$
\vep_* = \rho_* = \rho.
$$
Now, choose any other function $\psi$ satisfying $T(\psi) = 0$ and rotate $\mm_*,\mb_*$ by $\psi$; let $\{T,\mm_o,\mb_o\}$ denote the corresponding frame.  Then the analogue of \eqref{eqn:gauge2} for the frame $\{T,\mm_o,\mb_o\}$ shows that $\vep_o$ remains unchanged,
$$
\vep_o = \vep_* = \rho_* = \rho,
$$
so that our task would be complete if we can find a $\psi$ satisfying
\beqa
\label{eqn:char}
T(\psi) = 0 \comma \text{Re}(\beta_o) = 0 \comma T(\beta_o) = 0.
\eeqa
To do so, go back to the complex frame $\{T,\mm_*,\mb_*\}$ and observe that when $\vep_* = \rho_*$, then
\beqa
\label{eqn:LB0}
[\kk,\mm_*] \overset{\eqref{eqn:LB1}}{=} 0.
\eeqa
Let $\{T,X_*,Y_*\}$ denote the underlying orthonormal frame corresponding to $\{T,\mm_*,\mb_*\}$.  Since $[T,X_*] = 0$, there exist local coordinates $(t,u,v)$ and functions $p,q,r$ such that
$$
\kk = \partial_t \comma \xx_* = \partial_u \comma \yy_* =  p\partial_t+q\partial_u+r\partial_v,
$$
with $p,q,r$ functions of $u,v$ only, since $[T,Y_*] = 0$, and with $r$ nowhere vanishing. The  coframe metrically equivalent to $\{T,X_*,Y_*\}$ is therefore
$$
\kk^{\flat} = dt-\frac{p}{r}dv \comma X_*^{\flat} = du-\frac{q}{r}dv \comma \yy_*^{\flat} = \frac{1}{r}dv.
$$
Next, since $(X_*^{\flat})^2 + (Y_*^{\flat})^2$ defines a Riemannian metric on the 2-manifold with coordinates $\{(u,v)\}$, and since any Riemannian 2-manifold is locally conformally flat (see, e.g., \cite{chern}), it follows that there exist coordinates $(x,y)$ and a smooth function $\lambda(x,y)$ such that
$$
(X_*^{\flat})^2 + (Y_*^{\flat})^2 = e^{2\lambda}(dx^2+dy^2).
$$
By a rotation in $x, y$ if necessary, we may further assume that
$$
X_*^\flat = e^{\lambda}dx \comma Y_*^\flat = e^{\lambda}dy.
$$
In the new coordinates $(t,x,y)$, we thus have  that
$$
T = \partial_t \comma X_* = e^{-\lambda}(\partial_x+a\partial_t) \comma Y_* = e^{-\lambda}(\partial_y+b\partial_t),
$$
for some smooth functions $a(x,y), b(x,y)$.  With these coordinates in hand, we now return to the task of satisfying $\text{Re}(\beta_o) = T(\beta_o) = 0$ in \eqref{eqn:char}.  For the former, $\psi(x,y)$ should satisfy $\text{Re}(\beta_o) = \text{Re}\big(e^{i\psi}( \beta_* + i\mm_*(\psi))\big) = 0$, or
\beqa
\label{eqn:PDE}
e^{i\psi}( \beta_* + i\mm_*(\psi)) + e^{-i\psi}( \bar{\beta}_* - i\,\mb_*(\psi)) = 0.
\eeqa
When expanded, and using the fact that 
$$
\text{div}\,X_* = \frac{\lambda_x}{e^\lambda}  \comma \text{div}\,Y_* = \frac{\lambda_y}{e^\lambda},
$$
\eqref{eqn:PDE} is a quasilinear first-order PDE in $\psi$,
$$
(\sin\psi)\psi_x - (\cos \psi)\psi_y = (\cos\psi)\lambda_x + (\sin\psi) \lambda_y,
$$
which has a solution by the method of characteristics.  For the latter, \eqref{eqn:S3} and \eqref{eqn:S5} together yield
$$
T(\beta_o) - \underbrace{\,\mm(\vep_o)\,}_{\mm_o(\rho_o)} \overset{\eqref{eqn:S5}}{=} -\!\!\underbrace{\,\text{Ric}(T,\mm_o)\,}_{\mm_o(\rho_o)~\text{by}~\eqref{eqn:S3}} \imp T(\beta_o) = 0,
$$
completing the proof.
\end{proof}

The following Corollary collects together what we've established so far:

\begin{cor}
\label{cor:above}
Let $(M,g)$ be a Riemannian 3-manifold and  $T$ a unit length Killing vector field with twist function $\omega^2$.  Then there exists an orthonormal frame $\{T,X,Y\}$ satisfying
\beqa
\label{eqn:dgauge}
\kappa = \sigma = 0 \comma \rho = \vep = -\frac{i}{2}\omega \comma  \beta = -\frac{i}{\sqrt{2}}\,\emph{\text{div}}\,Y,
\eeqa
and with $T(\omega) = T(\beta) = 0$.  In this frame, \eqref{eqn:Ric} takes the form
\beqa
\label{eqn:Ric2}
Y(\emph{\text{div}}\,Y) = - (\emph{\text{div}}\,Y)^2 -\frac{1}{2}\Big(S + \frac{\omega^2}{2}\Big)\cdot
\eeqa
\end{cor}
Notice that \eqref{eqn:Ric2} implies that such a frame may not always exist \emph{globally}; e.g., if $M$ is compact and $S$ is nonnegative and positive somewhere, then a standard Riccati argument yields that in such a case the only complete solution to \eqref{eqn:Ric2} is one where $\text{div}\,Y = S + \frac{\omega^2}{2} = 0$, which is impossible.  We now proceed to our local classification.

\section{The Local Classification}
\label{sec:proof}
Theorem \ref{thm:main} follows from one further modification to the orthonormal frame satisfying \eqref{eqn:dgauge}:

\begin{customthm}{1}
\label{thm:main}
Let $(M,g)$ be a Riemannian 3-manifold that admits a unit length Killing vector field $T$.  Then there exist local coordinates $(t,r,\theta)$ and a smooth function $\varphi(r,\theta)$ such that
\beqa
\label{eqn:gT}
T = \partial_t \comma g = (T^{\flat})^2 + dr^2 + \varphi^2d\theta^2,
\eeqa
and where the quotient metric $dr^2 + \varphi^2d\theta^2$ has Gaussian curvature
\beqa
\label{eqn:gauss}
-\frac{\varphi_{rr}}{\varphi} = \frac{1}{2}\big(S + \emph{\text{Ric}}(T,T)\big),
\eeqa
with $S$ and $\emph{\text{Ric}}$ the scalar curvature and Ricci tensor of $g$, respectively.  If \eqref{eqn:gT0} is given globally on $M = \RR^3$ \emph{(}with $r,\theta$ polar coordinates on $\RR^2$\emph{)}, and if $\emph{\text{Ric}}(T,T) =  0, \varphi(0,\theta) = 0$, $\varphi_r(0,\theta)=1$, $\varphi(-r,\theta+\pi) = -\varphi(r,\theta)$, and $\varphi > 0$ when $r > 0$, then $g$ is complete if and only if
$$
\lim_{r\to \infty} \inf_{|p| \geq r} S\big|_p\leq 0.
$$
\end{customthm}

\begin{proof}
Let $(M,g)$ be a Riemannian 3-manifold and  $T$ a unit length Killing vector field with twist function $\omega^2$.  By Corollary \ref{cor:above}, there exist a local orthonormal frame $\{T,X,Y\}$ satisfying \eqref{eqn:dgauge} and coordinates $(t,x,y)$ in which $T = \partial/\partial t$.  Let $\{T^{\flat},X^{\flat},Y^{\flat}\}$ denote the dual coframe.  We now modify the coordinates $(t,x,y)$ while keeping $T = \partial/\partial t$ unchanged.  The key is that \eqref{eqn:LB1} and \eqref{eqn:LB2} satisfy
$$
[T,X] = [T,Y] = 0 \comma [X,Y] = \omega T + (\text{div}\,Y) X,
$$
from which it follows that $Y^{\flat}$ is closed, $dY^{\flat} = 0$; hence
$$
Y^{\flat}  = dr
$$
for some smooth function $r(x,y)$.  Similarly,
$$
dX^{\flat} = (\text{div}\,Y) Y^{\flat} \wedge X^{\flat} \imp X^{\flat} = \varphi d\theta
$$
for some smooth functions $\varphi(x,y)>0$ and $\theta(x,y)$, with the former satisfying
\beqa
\label{eqn:cr}
Y(\varphi) = (\text{div}\,Y)\varphi
\eeqa
(recall that $T(\beta)=0$). Since $X(r)=Y(\theta) = 0$, we can define new coordinates $(t,r,\theta)$, in terms of which the frame $\{T,X,Y\}$ takes the form
\beqa
\label{eqn:coord1}
T = \partial_t  \comma X = h\partial_t  + \frac{1}{\varphi}\partial_\theta \comma Y = k\partial_t + \partial_r,
\eeqa
for some smooth functions $h,k$; furthermore, $\varphi_t = h_t = k_t = 0$ (recall that $[T,X] = [T,Y] = 0$), so that $\varphi,h,k$ are all functions of $r,\theta$ only.  Thus
$$
g = (T^{\flat})^2 + (X^{\flat})^2 + (Y^{\flat})^2 = (T^{\flat})^2 + dr^2 + \varphi^2d\theta^2,
$$
confirming \eqref{eqn:gT}.  Now, the quotient metric $dr^2+\varphi^2d\theta^2$ has scalar curvature $-2\varphi_{rr}/\varphi$, hence Gaussian curvature $-\varphi_{rr}/\varphi$.  To relate this to the curvature of $(M,g)$, we take a $Y$-derivative of \eqref{eqn:cr}, make use of \eqref{eqn:Ric2}, and note that $\partial_t(\text{div}\,Y) = 0$ by \eqref{eqn:guage}, to obtain
$$
\varphi_{rr} = Y(\text{div}\,Y)\varphi + (\text{div}\,Y)^2\varphi \imp -\frac{\varphi_{rr}}{\varphi} \overset{\eqref{eqn:Ric2}}{=}\frac{1}{2}\Big(S + \frac{\omega^2}{2}\Big)\cdot 
$$
Since $\text{Ric}(T,T) =  \frac{\omega^2}{2}$ by Lemma \ref{lemma:KVF2}, this confirms \eqref{eqn:gauss}.  There remains, finally, the question of completeness; thus, suppose that on $\RR^3 = \{(t,r,\theta)\}$ with $r,\theta$ polar coordinates on $\RR^2$, and with metric $g$ given by \eqref{eqn:gT}, we have the globally defined vector fields appearing in \eqref{eqn:coord1}.  If $\text{Ric}(T,T) = \frac{\omega^2}{2} = 0$, so that $g(T,[X,Y]) = 0$, then $T= \partial_t$ is parallel in $(\RR^3,g)$ and thus $T^{\flat} = dt$, from which it follows from \eqref{eqn:coord1} that $h=k = 0$, hence that $\{\partial_t,\partial_r,\partial_\theta\}$ is orthogonal.  Now, let $\gamma(s) = (t(s),r(s),\theta(s))$ be a geodesic in $(\RR^3,g)$; since $g(T,\gamma') = g(\partial_t,\gamma')$ is a constant, which constant we denote by $c$, the tangent vector $\gamma'(s)$ takes the form
$$
\gamma'(s) = c\partial_t|_{\gamma(s)} + \dot{r}(s)\partial_r|_{\gamma(s)} + \dot{\theta}(s)\partial_\theta|_{\gamma(s)}.
$$
Since $g$ splits as the product of $dt^2$ and $dr^2 + \varphi^2d\theta^2$, $\gamma(s)$ is a geodesic if and only if $(r(s),\theta(s))$ is a geodesic in $(\RR^2,dr^2+\varphi^2d\theta^2)$.  We now use \cite{kw2}: when $r,\theta$ are polar coordinates, with $\varphi(0,\theta) = 0, \varphi_r(0,\theta)=1$, $\varphi(-r,\theta+\pi) = -\varphi(r,\theta)$, and $\varphi > 0$ when $r > 0$, then the metric is complete if and only if 
$$
\lim_{r\to \infty} \inf_{|p| \geq r} \Big(\!\!-\!\frac{\varphi_{rr}}{\varphi}\Big)\Big|_p \leq 0.
$$
Since $-\frac{\varphi_{rr}}{\varphi} = \frac{1}{2}\big(S + \text{Ric}(T,T)\big) = \frac{1}{2}S$, the proof is complete.
\end{proof}

A final remark regarding Theorem \ref{thm:main}: bear in mind that, since in general
$$
T^{\flat} = dt -\varphi hd\theta -kdr,
$$
the coordinates $(t,r,\theta)$ above are not ``semigeodesic" (see, e.g., \cite{Lee}); indeed, the metric components $g_{ij}$ in the coordinate basis $\{\partial_t,\partial_r,\partial_\theta\}$ are given by
\beqa
\label{eqn:gtr0}
(g_{ij}) = \begin{pmatrix}
1 & -k & -\varphi h\\
-k & 1+k^2 & \varphi hk\\
 -\varphi h & \varphi h k & \varphi^2(1+h^2)
\end{pmatrix}\cdot
\eeqa
We now proceed to the Lorentzian setting.

\section{The Lorentzian setting}
\label{sec:Lor}
Before proceeding to a proof of the Lorentzian analogue of Theorem \ref{thm:main}, we first collect a few facts from Lorentzian geometry; in what follows we adopt the metric index $(-\!+\!+)$.  First, a vector field $T$ on a Lorentzian manifold $(M,\gL)$ is \emph{timelike} if $\gL(T,T) < 0$.  Second, if a timelike $T$ has unit length, $\gL(T,T) = -1$, then
\beqa
\label{eqn:gR2}
\gR \defeq \gL + 2(T^{\flat_{\scriptscriptstyle L}})^2
\eeqa
defines a Riemannian metric on $M$ (here $T^{\flat_{\scriptscriptstyle L}} = \gL(T,\cdot)$).  Third, the following properties hold between $\gR$ and $\gL$:
\begin{enumerate}[leftmargin=*]
\item[1.] $T$ is a unit length Killing vector field with respect to $\gR$ if and only if $T$ is a unit timelike Killing vector field with respect to $\gL$ (see, e.g., \cite{olea}).
\item[2.] If $T$ is a $\gR$-unit length Killing vector field, then $$\text{Ric}_{\scriptscriptstyle R}(T,T) = \text{Ric}_{\scriptscriptstyle L}(T,T)$$ (consult \cite{olea}; this follows because $\nabla^{\scriptscriptstyle R}_{\!X}T = -\nabla^{\scriptscriptstyle L}_{\!X}T$ for any unit length $X$ that is $\gR$- or $\gL$-orthogonal to $T$, where $\nabla^{\scriptscriptstyle R}$ and $\nabla^{\scriptscriptstyle L}$ are, respectively, the Levi-Civita connections of $\gR$ and $\gL$), while their scalar curvatures $S_{\scriptscriptstyle R}$ and $S_{\scriptscriptstyle L}$ satisfy
$$
S_{\scriptscriptstyle L} = S_{\scriptscriptstyle R} + 2\text{Ric}_{\scriptscriptstyle R}(T,T).
$$
In particular, $S_{\scriptscriptstyle R} + \text{Ric}_{\scriptscriptstyle R}(T,T) = S_{\scriptscriptstyle L} - \text{Ric}_{\scriptscriptstyle L}(T,T)$.
\item[3.] If $T$ is $\gR$-unit length Killing vector field, then $\gL$ is complete if and only if $\gR$ is complete (see \cite{romero}).
\end{enumerate}

With these facts established, the Lorentzian analogue of Theorem \ref{thm:main} now follows easily:

\begin{cor}
\label{thm:main2}
Let $(M,\gL)$ be a Lorentzian 3-manifold that admits a unit timelike Killing vector field $T$.  Then there exists local coordinates $(t,r,\theta)$ and a smooth function $\varphi(r,\theta)$ such that
\beqa
\label{eqn:gTL}
T = \partial_t \comma \gL = -(T^{\flat})^2 + dr^2 + \varphi^2d\theta^2,
\eeqa
and where the quotient metric $dr^2 + \varphi^2d\theta^2$ has Gaussian curvature
$$
-\frac{\varphi_{rr}}{\varphi} = \frac{1}{2}\big(S_{\scriptscriptstyle L} - \emph{\text{Ric}}_{\scriptscriptstyle L}(T,T)\big),
$$
with $S_{\scriptscriptstyle L}$ and $\emph{\text{Ric}}_{\scriptscriptstyle L}$ the scalar curvature and Ricci tensor of $\gL$, respectively.  Furthermore, $\gL$ is complete if and only if $\gR$ is complete, where $\gR$ is the corresponding Riemannian metric given by \eqref{eqn:gR2}.
\end{cor}

\begin{proof}
By our remarks above, $T$ is a unit length Killing vector field with respect to the Riemannian metric $\gR$, with $S_{\scriptscriptstyle R} + \text{Ric}_{\scriptscriptstyle R}(T,T) = S_{\scriptscriptstyle L} - \text{Ric}_{\scriptscriptstyle L}(T,T)$; Corollary \ref{thm:main2} therefore follows immediately from Theorem \ref{thm:main}.
\end{proof}

\section{The compact case}
We now prove a global obstruction result in the compact setting.  Thus, let $(M,g)$ be a compact Riemannian 3-manifold equipped with a globally defined unit length Killing vector field. %with \emph{constant} nonzero twist $\omega^2$.  
With respect to a local orthonormal frame $\{T,X,Y\}$, we have, by Lemma \ref{lemma:KVF2}, that
$$
\underbrace{\,\text{Ric}(X,X) = \text{Ric}(Y,Y)\,}_{\text{via}~\text{Ric}(\mm,\mm)\,=\,0} \comma \text{Ric}(T,T) = \frac{\omega^2}{2}.
$$
 In fact, because
$$
\frac{S}{2} - \frac{\omega^2}{4} \overset{\eqref{eqn:S}}{=} \text{Ric}(\mm,\mb) = \frac{1}{2}\big(\text{Ric}(\xx,\xx) + \text{Ric}(\yy,\yy)\big),
$$
it follows that $\text{Ric}(X,X) = \text{Ric}(Y,Y) = \frac{S}{2} - \frac{\omega^2}{4}$.  
Finally, by \eqref{eqn:S3} we get $\text{Ric}(T,\mm)=\mm(\rho)$, whose real and imaginary parts yield
\beqa
\label{eqn:RR1}
\text{Ric}(T,X)=-\frac{Y(\omega)}{2} \comma \text{Ric}(T,Y)=\frac{Y(\omega)}{2}\cdot
\eeqa

Thus the Ricci operator $\text{Ric}\colon TM \lra TM$, defined by
\beqa
\label{eqn:Ricop}
v \mapsto \text{Ric}(v) = R(v,T)T + R(v,X)X + R(v,Y)Y,
\eeqa
has, with respect to the frame $\{T,X,Y\}$, the matrix
\beqa
\label{eqn:Ham1}
\text{Ric} = \frac{1}{2}\begin{pmatrix}
\omega^2 & -Y(\omega) & X(\omega)\\
-Y(\omega) & S - \frac{\omega^2}{2} & 0\\
X(\omega) & 0 & S - \frac{\omega^2}{2}
\end{pmatrix}\cdot
\eeqa
The characteristic polynomial of $\text{Ric}$ is
$$
\label{eqn:charpol}
\left(\frac{S}{2}-\frac{\omega^2}{4}-\lambda\right)\left(\lambda^2-\left(\frac{S}{2}+\frac{\omega^2}{4}\right)\lambda+\frac{\omega^2}{2}\left(\frac{S}{2}-\frac{\omega^2}{4}\right)-\frac{1}{4}|\nabla\omega|_g^2\right),
$$
where $\nabla \omega = X(\omega)X+Y(\omega)Y$ is the gradient of $\omega$; the eigenvalues of $\text{Ric}$ are then easily found to be
\beqa \label{eqn:eigenval}
\lambda_1 = \frac{S}{4}+\frac{\omega^2}{8} + \frac{\sqrt{\Delta}}{2} \comma \lambda_2 = \frac{S}{4}+\frac{\omega^2}{8} - \frac{\sqrt{\Delta}}{2} \comma \lambda_3=\frac{S}{2}-\frac{\omega^2}{4},
\eeqa
where $\Delta \defeq \frac{1}{4}(S-\frac{3}{2}\omega^2)^2+|\nabla\omega|_g^2$.  Note that when the twist is constant, \eqref{eqn:Ham1} reduces to
\beqa
\label{eqn:Ham0}
\text{Ric} = \begin{pmatrix}
\frac{\omega^2}{2} & 0 & 0\\
0 & \frac{S}{2} - \frac{\omega^2}{4} & 0\\
0 & 0 & \frac{S}{2} - \frac{\omega^2}{4}
\end{pmatrix}\cdot
\eeqa
 As mentioned in the Introduction, the canonical such example is $(\mathbb{S}^3,\mathring{g})$ with radius $R$ and Hopf Killing vector field $T$:
\beqa
\label{eqn:S3K}
\text{Ric}(T,T)  = \frac{2}{R^2} = \frac{\omega^2}{2} \comma S = \frac{6}{R^2}  \imp \frac{S}{2} - \frac{\omega^2}{4} = \frac{\omega^2}{2}\cdot
\eeqa
In any case, owing to Hamilton's well known result regarding the positivity of the Ricci operator in dimension 3 \cite{Hamilton}, we have the following global obstruction:

\begin{customthm}{2}
\label{thm:2}
Let $(M,g)$ be a compact Riemannian 3-manifold and $T$ a globally defined, unit length Killing vector field.  If $\emph{\text{Ric}}(T,T) \neq 0$ and
\beqa
\label{eqn:pos}
S > 2\frac{|\emph{\text{Ric}}(T)|_g^2}{\emph{\text{Ric}}(T,T)}-\emph{\text{Ric}}(T,T),
\eeqa
then $M$ admits a metric of constant positive sectional curvature.
\end{customthm}

\begin{proof}
Observe that the eigenvalues of $\text{Ric}$ in \eqref{eqn:eigenval} are all positive when 
\beqa
\label{eqn:Sprelim}
S>\frac{|\nabla\omega|_g^2}{\omega^2}+\frac{\omega^2}{2}\cdot
\eeqa
Because
\beqa
\text{Ric}(T) \!\!&=&\!\! R(T,T)T + R(T,X)X + R(T,Y)Y\nonumber\\
\!\!&\overset{\eqref{eqn:RR1}}{=}&\!\! \frac{\omega^2}{2}T  -  \frac{Y(\omega)}{2}X + \frac{X(\omega)}{2}Y,\label{eqn:Ric22}
\eeqa
it follows that
$$
|\text{Ric}(T)|_g^2 = \frac{1}{4}(\omega^4 + |\nabla\omega|_g^2).
$$
Since $\text{Ric}(T,T)=\frac{\omega^2}{2}$, \eqref{eqn:pos} implies \eqref{eqn:Sprelim}.  By \cite{Hamilton}, positive Ricci operator implies that $M$ admits a metric of constant positive sectional curvature.
\end{proof}
Note that the positive sectional curvature condition in \cite{Hamilton} requires that no eigenvalue of $\text{Ric}$ should be larger than the sum of the other two eigenvalues. This requires $S$ twice as large: $S>2\frac{|\nabla\omega|_g^2}{\omega^2}+\omega^2$.

\section{Criterion for Conformal Flatness}
A metric on a 3-manifold is locally conformally flat if and only if its Cotton-York tensor vanishes; since this 2-tensor is symmetric and trace-free, this gives five conditions.  The Cotton-York tensor is calculated in Appendix \ref{app:CY}, where it is written in matrix form as $\begin{pmatrix} c_1 & c_2 & c_3\end{pmatrix}$ with respect to a local orthonormal frame $\{T,X,Y\}$ satisfying \eqref{eqn:coord1} in the coordinates $(t,r,\theta)$; see \eqref{eqn:Y1}-\eqref{eqn:Y3} below.  In what follows, the entry in the $i^{\text{th}}$ column and $j^{\text{th}}$ row is denoted by $c_{ij}$.  With that said, we now proceed to the proof of Theorem \ref{thm:3}:

\begin{customthm}{3}
\label{thm:3}
Let $(M,g)$ be a Riemannian 3-manifold that admits a unit length Killing vector field $T$.  If $g$ is locally conformally flat, then
\beqa
\label{eqn:wPDE2}
4|\emph{\text{Ric}}(T)|_g^2 = 3\emph{\text{Ric}}(T,T)^2 - 2B\emph{\text{Ric}}(T,T) + C
\eeqa
for some constants $B,C$, where $\emph{\text{Ric}}(T)$ is the Ricci operator.
Conversely, given \eqref{eqn:wPDE2}, there exist coordinates $(r,\theta)$ on the quotient metric in \eqref{eqn:gT0} with respect to which $g$ is conformally flat when 
\[ \omega_\theta=0 \comma \omega_r^2 + \frac{1}{4}(\omega^2 + 2B)^2= C+B^2 \comma \varphi=h(\theta)\omega_r, \]
where $\omega^2  = 2\emph{\text{Ric}}(T,T)$, $\varphi$ is as in Theorem \ref{thm:main}, and $h(\theta)$ is a smooth function.  If $\emph{\text{Ric}}(T,T)$ is constant, then  $g$ is locally conformally flat if and only if $S = 3\emph{\text{Ric}}(T,T)$.
\end{customthm}

\begin{proof}
We start by setting the entry $c_{32}$ equal to zero,
\[ \frac{1}{2}Y(X(\omega)) = \frac{1}{2}\frac{\partial}{\partial r}\left(\frac{1}{\varphi}\frac{\partial\omega}{\partial\theta}\right)=0\]
(recall from Lemma \ref{lemma:KVF2} that $\partial_t\omega  = T(\omega) = 0$), which implies that 
\begin{equation}\label{omeq}
\omega_\theta = A(\theta)\varphi
\end{equation} for some function $A(\theta)$. Next, $c_{21} = c_{31} = 0$ together yield
\begin{equation}
\label{Seq}S=\frac{5}{2}\omega^2+B_1 = 5\text{Ric}(T,T) + B_1
\end{equation} for some constant $B_1$.
It follows that $\frac{1}{2}\big(S+\text{Ric}(T,T)\big)=-\frac{\varphi_{rr}}{\varphi} = \frac{B_1}{2}+\frac{3}{2}\omega^2$.
The remaining two conditions are $c_{22} = c_{33} = 0$.
Substituting \eqref{omeq} and \eqref{Seq} into $c_{33}=0$, and recalling \eqref{eqn:cr}, gives
\[ \omega(\omega^2+B_1)=-2\frac{\omega_r \varphi_r + A'(\theta)}{\varphi}, \]
which, after rearranging, becomes
\begin{equation}
\label{y33}
\varphi\,\omega(\omega^2+B_1)+2\omega_r \varphi_r= -2A'(\theta).
\end{equation}
Finally, from $c_{22} = 0$ we get
\begin{equation*}\label{y22} \omega(\omega^2+B_1)=-2\omega_{rr},\end{equation*}
and after multiplying through by $-\omega_r$ and integrating yields
\beqa
\label{eqn:f}f(\theta) = \omega^4 + 2B_1\omega^2 + 4\omega_r^2
\eeqa
for some function $f(\theta)$. To relate $f(\theta)$ and $A(\theta)$, take a $\theta$-derivative of $f$,
\begin{align*}f'(\theta)&=4\omega^3\omega_\theta+4B_1\omega\omega_\theta+4\omega_r\omega_{r\theta},\\
&\overset{\eqref{omeq}}{=}4\omega^3A(\theta)\varphi+4B_1\omega A(\theta)\varphi+8\omega_r A(\theta)\varphi_r,\\
&=4A(\theta)\big(\varphi\omega(\omega^2+B_1)+2\omega_r \varphi_r\big),\\
&\overset{\eqref{y33}}{=}-8A(\theta)A'(\theta),
\end{align*}
and integrate, to obtain
\[ f(\theta)=-4A^2(\theta)+4C\]
for some constant $C$. Inserting this back into \eqref{eqn:f} gives
$$
4\omega_r^2 = -4A^2(\theta) + 4C + B_1^2 - (\omega^2+B_1)^2.
$$
Substituting \eqref{omeq} for $A(\theta)$, dividing through by $4$, and setting $B \defeq B_1/2$, yields
\beqa
\label{eqn:prelim1}
\omega_r^2 + \frac{\omega_{\theta}^2}{\varphi^2} = C - \frac{\omega^4}{4} - B\omega^2.
\eeqa
The left-hand side can be further simplified; indeed, since $\text{Ric}(T,T) = \frac{\omega^2}{2}$,
\beqa
\text{Ric}(T)
\!\!&\overset{\eqref{eqn:coord1}}{=}&\!\! \frac{\omega^2}{2}T - \frac{\omega_r}{2}X + \frac{\omega_\theta}{2\varphi}Y\nonumber
\eeqa
(recall \eqref{eqn:Ric22}), so that
$$
|\text{Ric}(T)|_g^2 = \frac{\omega^4}{4} + \frac{\omega_r^2}{4} + \frac{\omega_\theta^2}{4\varphi^2}\cdot
$$
Substituting this into \eqref{eqn:prelim1} yields
$$
4|\text{Ric}(T)|_g^2 = 3\text{Ric}(T,T)^2 - 2B\text{Ric}(T,T) + C,
$$ 
which is precisely \eqref{eqn:wPDE2}.  Conversely, suppose that \eqref{eqn:wPDE2} holds; then \eqref{eqn:prelim1} holds and we see that $|d\omega|_g^2=C-\frac{\omega^4}{4}-B\omega^2=|\nabla\omega|_g^2$.  Next, observe that the vector field
$$
\widetilde{X}\defeq2\text{Ric}(T)-\omega^2T=X(\omega)Y-Y(\omega)X
$$
is divergence-free and satisfies both $|\widetilde{X}|_g=|\nabla\omega|_g$ and $g(\widetilde{X}, \nabla\omega)=0$, in which case its normalization will also be divergence-free:
\[ \text{div}\bigg(\frac{\widetilde{X}}{|\widetilde{X}|_g}\bigg) = -\frac{g(\nabla|\widetilde{X}|_g,\widetilde{X})}{|\widetilde{X}|_g^2}=0. \]
Then, setting $\widetilde{Y}\defeq\frac{\nabla\omega}{|\nabla\omega|_g}$ gives an orthonormal frame $\{T,\widetilde{X},\widetilde{Y}\}$ satisfying \eqref{eqn:guage}.
Working in this frame, set $X\defeq\widetilde{X}$, $Y\defeq\widetilde{Y}$ and adjust the coordinates $r,\theta$ accordingly. Then in these new
coordinates $\omega_\theta=0$ and \eqref{eqn:prelim1} becomes the following ODE:
\[ \omega_r^2 + \frac{1}{4}(\omega^2 + 2B)^2= C+B^2. \]
This has the form of a conservation of energy equation with positive potential.
The potential is a single well when $B\ge0$ and a double well when $B<0$.
Thus there will be periodic solutions for generic constants $B$, $C$ and initial value $\omega|_{r=0}=\omega_0$ satisfying $C+B^2\ge\frac{1}{4}(\omega_0^2+2B)^2$.  This is not enough to guarantee conformal flatness, as \eqref{Seq} and \eqref{y33} must also be satisfied.
In light of \eqref{eqn:gauss}, we now show that these require that $\varphi=h(\theta)\omega_r$ for some function $h(\theta)$. Indeed, taking an $r$ derivative of \eqref{eqn:prelim1} yields
\[ 2\omega_{rr}\omega_{r}=-\omega(\omega^2+B_1)\omega_r. \]
Since \eqref{omeq} implies that $A$ is zero, the above implies that \eqref{y33} can be written as
\[-2\omega_{rr}\varphi+2\omega_r\varphi_r=0.\]
This requires $\omega$ constant or $\varphi=h(\theta)\omega_r$ for some function $h(\theta)$.
Taking two $r$ derivatives yields
\beqa
\varphi_{rr} \!\!&=&\!\! h(\theta)\omega_{rrr}\nonumber \\
\!\!&=&\!\! h(\theta)\left(-\frac{3}{2}\omega^2-\frac{B_1}{2}\right)\omega_r\nonumber 
\eeqa
Now using \eqref{eqn:gauss}, this gives \eqref{Seq}, showing that for this choice of $\phi$, the metric $g$ is conformally flat.  Finally, as \eqref{eqn:CFconst} in Appendix \ref{app:CY} shows, if $\text{Ric}(T,T)$ is constant, then the Cotton-York tensor vanishes if and only if the scalar curvature satisfies $S = 3\text{Ric}(T,T)$.
\end{proof}

\begin{appendices}
\section{Derivation of the Cotton-York Tensor}
\label{app:CY}
In order to compute the Cotton tensor, we will use the Cartan formalism (see, e.g., \cite[pp.~111-12]{PP}).  First recall the \emph{Schouten tensor} $P$: 
\[ P \defeq \mathrm{Ric}-\frac{S}{4}g. \]
Using the fact that $P\in\Omega^1(TM)$, the Cotton tensor is given by 
\[ \mathrm{Cot}(U,V) \defeq d^\nabla P(U,V) = (\nabla_U P)(V)-(\nabla_V P)(U). \]
We now work in a frame $\{T,X,Y\}$ satisfying \eqref{eqn:dgauge}; with respect to it, the Levi-Civita connection form $\boldsymbol{\omega}$ (not to be  confused with the twist function $\omega$) is 
\[ \boldsymbol{\omega} = \begin{pmatrix}
0 & -\frac{\omega}{2}Y^\flat & \frac{\omega}{2}X^\flat \\
\frac{\omega}{2}Y^\flat & 0 & \frac{\omega}{2}T^\flat+\mathrm{div}(Y)X^\flat \\
-\frac{\omega}{2}X^\flat & -\frac{\omega}{2}T^\flat-\mathrm{div}(Y)X^\flat & 0\\
\end{pmatrix}\cdot \]
Together with the Ricci tensor \eqref{eqn:Ham1}, the Schouten tensor is thus
\[ P = \frac{1}{2}\begin{pmatrix}
\omega^2-\frac{S}{2} & -Y(\omega) & X(\omega)\\
-Y(\omega) & \frac{S}{2} - \frac{\omega^2}{2} & 0\\
X(\omega) & 0 & \frac{S}{2} - \frac{\omega^2}{2}
\end{pmatrix}\cdot
\]
The frame formula for the exterior covariant derivative of $P$ is just
\[ d^\nabla P = dP + [\boldsymbol{\omega},P].\]
Using this, and recalling that $S$ and $\omega$ have zero derivative in the direction of $T$, we have
\begin{align*}
\mathrm{Cot}(X,Y)&=(X(P)+[\boldsymbol{\omega}(X),P])\left(\begin{array}{c} 0\\0\\1 \end{array}\right)-(Y(P)+[\boldsymbol{\omega}(Y),P])\left(\begin{array}{c} 0\\1\\0 \end{array}\right) \\
&= \left(\begin{array}{c} - \frac{3}{4}\omega^3+\frac{1}{2}S\omega+\frac{1}{2}\mathrm{div}(Y)Y(\omega)+\frac{1}{2}[X(X(\omega))+Y(Y(\omega))] \\ -\frac{1}{4}Y(S)+\frac{5}{4}\omega Y(\omega) \\ \frac{1}{4}X(S)-\frac{5}{4}\omega X(\omega)\end{array}\right),\\
\mathrm{Cot}(Y,T)&=(Y(P)+[\boldsymbol{\omega}(Y),P])\left(\begin{array}{c} 1\\0\\0 \end{array}\right)-[\boldsymbol{\omega}(T),P]\left(\begin{array}{c} 0\\0\\1 \end{array}\right) \\
&= \left(\begin{array}{c} \frac{5}{4}\omega Y(\omega)-\frac{1}{4}Y(S) \\ \frac{3}{8}\omega^3-\frac{1}{4}S\omega-\frac{1}{2}Y(Y(\omega)) \\ \frac{1}{2}Y(X(\omega))\end{array}\right),\\
\mathrm{Cot}(T,X)&=[\boldsymbol{\omega}(T),P]\left(\begin{array}{c} 0\\1\\0 \end{array}\right)-(X(P)+[\boldsymbol{\omega}(X),P])\left(\begin{array}{c} 1\\0\\0 \end{array}\right) \\
&= \left(\begin{array}{c} \frac{1}{4}X(S)-\frac{5}{4}\omega  X(\omega) \\ \frac{1}{2}X(Y(\omega)) -\frac{1}{2}X(\omega)\mathrm{div}(Y) \\ \frac{3}{8}\omega^3-\frac{1}{4}S\omega-\frac{1}{2}Y(\omega)\mathrm{div}(Y)-\frac{1}{2}X(X(\omega))\end{array}\right)\cdot
\end{align*}
The Cotton-York tensor is the Hodge-star of the Cotton tensor:
\begin{align*}
CY &\defeq \star \mathrm{Cot} .
\end{align*}
Written as $CY=\left(\begin{array}{ccc} c_1 & c_2 & c_3 \end{array}\right)$ with columns $c_1,c_2,c_3$, it is given by
\begin{align}
c_1&=\left(\begin{array}{c} - \frac{3}{4}\omega^3+\frac{1}{2}S\omega+\frac{1}{2}\mathrm{div}(Y)Y(\omega)+\frac{1}{2}[X(X(\omega))+Y(Y(\omega))] \\ -\frac{1}{4}Y(S)+\frac{5}{4}\omega Y(\omega) \\ \frac{1}{4}X(S)-\frac{5}{4}\omega X(\omega)\end{array}\right),\label{eqn:Y1} \\
c_2&=\left(\begin{array}{c} \frac{5}{4}\omega Y(\omega)-\frac{1}{4}Y(S) \\ \frac{3}{8}\omega^3-\frac{1}{4}S\omega-\frac{1}{2}Y(Y(\omega)) \\ \frac{1}{2}Y(X(\omega))\end{array}\right),\label{eqn:Y2} \\
c_3&= \left(\begin{array}{c} \frac{1}{4}X(S)-\frac{5}{4}\omega X(\omega) \\ \frac{1}{2}Y(X(\omega)) \\ \frac{3}{8}\omega^3-\frac{1}{4}S\omega-\frac{1}{2}Y(\omega)\mathrm{div}(Y)-\frac{1}{2}X(X(\omega))\end{array}\right)\cdot\label{eqn:Y3}
\end{align}
Observe that if $\text{Ric}(T,T) = \omega^2/2$ is constant, then the Cotton-York tensor simplifies to
\begin{align*}
c_1&=\left(\begin{array}{c} - \frac{3}{4}\omega^3+\frac{1}{2}S\omega \\ -\frac{1}{4}Y(S)\\ \frac{1}{4}X(S)\end{array}\right), \\
c_2&=\left(\begin{array}{c} -\frac{1}{4}Y(S) \\ \frac{3}{8}\omega^3-\frac{1}{4}S\omega \\ 0\end{array}\right), \\
c_3&= \left(\begin{array}{c} \frac{1}{4}X(S) \\ 0 \\ \frac{3}{8}\omega^3-\frac{1}{4}S\omega\end{array}\right),
\end{align*}
in which case $CY = 0$ if and only if
\beqa
\label{eqn:CFconst}
S = 3\text{Ric}(T,T).
\eeqa
Of related interest is the case when $CY$ equals the traceless Ricci tensor, 
\beqa
\label{eqn:TMG}
CY = \text{Ric} - \frac{1}{3}Sg,
\eeqa
specifically the case when the scalar curvature $S$ is \emph{constant}; see, e.g., \cite{NC}, where this equality is related to so-called \emph{topological massive gravity} in dimension 3, and where $S = 6\Lambda$ with $\Lambda$ the cosmological constant.   We mention here in passing that in the presence of a unit length Killing vector field $T$, the condition \eqref{eqn:TMG} with $S$ constant implies that $\text{Ric}(T,T) = \frac{\omega^2}{2}$ is also constant.  Indeed,
\beqa
\frac{5}{4}\omega Y(\omega) \!\!\!&\overset{\eqref{eqn:Y1}}{=}&\!\!\! c_{12} \overset{\eqref{eqn:TMG}}{=} \text{Ric}(T,X) \overset{\eqref{eqn:S3}}{=} -\frac{Y(\omega)}{2},\nonumber\\
-\frac{5}{4}\omega X(\omega) \!\!\!&\overset{\eqref{eqn:Y1}}{=}&\!\!\! c_{13} \overset{\eqref{eqn:TMG}}{=} \text{Ric}(T,Y) \overset{\eqref{eqn:S3}}{=} \frac{X(\omega)}{2},\nonumber\\
\frac{1}{2}Y(X(\omega)) \!\!\!&\overset{\eqref{eqn:Y3}}{=}&\!\!\! c_{32} \overset{\eqref{eqn:TMG}}{=} \text{Ric}(X,Y) \overset{\eqref{eqn:S2}}{=} 0,\nonumber
\eeqa
together imply that  $\omega$ is constant, as can  be easily verified.

\end{appendices}
\bibliographystyle{alpha}
\bibliography{constant_KVF}
\end{document}